\documentclass[a4paper,12pt,reqno]{amsart} 

\makeatletter
\tagsleft@false 
\makeatother

\usepackage{comment}

\usepackage{tikz}
\usepackage{pgfplots}
\pgfplotsset{compat=1.18}
\usepackage{longtable}
\usepackage{amsmath}
\usepackage{amssymb}
\usepackage{amsthm}
\usepackage{mathrsfs}
\usepackage{ifthen}
\usepackage{graphicx}
\usepackage[T1]{fontenc}

\numberwithin{equation}{section}

\newtheorem*{theorem*}{Theorem}

\newtheorem{thm}{Theorem}[section]
\newtheorem{lem}{Lemma}[section]

\newtheorem*{cor*}{Corollary}

\theoremstyle{definition}

\newtheorem{prob}{Problem}[section]

\newtheorem{cor}{Corollary}[section]

\newcounter{minutes}
\divide\time by 60
\newcounter{hours}
\multiply\time by 60
\addtocounter{minutes}{-\time}

\newcounter{own}
\def\theown{\thesection.\arabic{own}}

\newenvironment{pf}[1][]{%
	\vskip 3mm
	\noindent
	\ifthenelse{\equal{#1}{}}%
	{{\slshape Proof. }}%
	{{\slshape #1.} }%
}%
{\qed\bigskip}

\newcounter{alphabet}

\def\be{\begin{equation}}
	\def\ee{\end{equation}}

\newcommand{\bee}{\begin{enumerate}}
	\newcommand{\eee}{\end{enumerate}}

\newcommand{\blem}{\begin{lem}}
	\newcommand{\elem}{\end{lem}}
\newcommand{\bthm}{\begin{thm}}
	\newcommand{\ethm}{\end{thm}}
\newcommand{\bcor}{\begin{cor}}
	\newcommand{\ecor}{\end{cor}}
\newcommand{\eeg}{\end{examp}}

\newcommand{\eegs}{\end{examples}}
\newcommand{\edefe}{\end{defin}}
\newcommand{\bprob}{\begin{prob}}
\newcommand{\eprob}{\end{prob}}
\newcommand{\bei}{\begin{itemize}}
\newcommand{\eei}{\end{itemize}}

\usepackage{xcolor}

\begin{document}

\title{{The second and third Hankel determinants for certain classes of functions}}

\author{Vasudevarao Allu}
\address{Vasudevarao Allu,
Department of Mathematics,
School of Basic Sciences,
Indian Institute of Technology Bhubaneswar,
Bhubaneswar-752050, Odisha, India.}
\email{avrao@iitbbs.ac.in}

\author{Shobhit Kumar}
\address{Shobhit Kumar,
Department of Mathematics,
School of Basic Sciences,
Indian Institute of Technology Bhubaneswar,
Bhubaneswar-752050, Odisha, India.}
\email{a21ma09007@iitbbs.ac.in}



\def\thefootnote{}
\footnotetext{
{ }
}
\makeatletter\def\thefootnote{\@arabic\c@footnote}\makeatother
\begin{abstract}
Let  $\mathcal{A}$ denote the class of analytic functions such that $f(0)=0$ and $f'(0)=1$ in the unit disk $\mathbb{D}:=\{z \in \mathbb{C}: |z|<1\}.$
In this paper, we consider  $\mathcal{S}^*(\varphi) := \left\{ f \in \mathcal{A} : zf'(z)/f(z) \prec \varphi(z):=(1+z/2)^2 \right\}$, a subclass of starlike functions and we compute the sharp second and third Hankel determinants for the functions in $\mathcal{S}^*(\varphi)$. Furthermore, we determine the extremal functions for the coefficient bounds of  the functions belonging to $\mathcal{S}^*(\varphi)$.
\end{abstract}
\maketitle
\pagestyle{myheadings}
\markboth{V. Allu and S. Kumar}{ The second and third Hankel determinants for certain classes of functions }

\section{Introduction}
Let $\mathcal{A}$ denote the class of normalized analytic functions $f(z)$ defined in the open unit disk $\mathbb{D} := \{ z \in \mathbb{C} : |z| < 1 \}$, with $f(0)=0$ and $f'(0)=1$. Then each $f \in \mathcal{A}$ has the representation
$
f(z) = z + \sum_{n=2}^{\infty} a_n z^n .
$
Let $\mathcal{S}$ be the subclass of $\mathcal{A}$ consisting of univalent (\textit{i.e.}, one-to-one) functions. Let $\mathcal{P}$ denote the class of functions with positive real part, having the series expansion
\[
p(z) = 1 + \sum_{n=1}^{\infty} p_n z^n .
\]
Let $f$ and $g$ be two analytic functions in the unit disk $%
\mathbb{D}$. Then $f$ is said to be subordinate to $g$, written as $f\prec g$
or $f(z)\prec g(z)$, if there exists a function $w(z)$, analytic in $%
\mathbb{D}$ with $w(0)=0$, $|w(z)|<1$ such that $f(z)=g(w(z))$ for $z\in%
\mathbb{D}$. Moreover, if $g$ is univalent in $\mathbb{D}$ and $f(0)=g(0)$,
then $f(\mathbb{D})\subseteq g(\mathbb{D})$.\\[2mm]
Let $\mathcal{S}^*$ be the subclass of $\mathcal{S}$ consisting of starlike functions that map $\mathbb{D}$ onto a starlike domain. Geometrically, $f \in \mathcal{S}^*$ if, and only if,
\[
\operatorname{Re}\, \!\left( \frac{z f'(z)}{f(z)} \right) > 0 \quad \mbox{for all } z \in \mathbb{D}.
\]
We define the class $\mathcal{S}^*(\varphi)$ such that
\[
\mathcal{S}^*(\varphi) := \left\{ f \in \mathcal{A} : \frac{z f'(z)}{f(z)} \prec \varphi(z) \right\},
\]
where $\varphi$ is analytic and univalent in $\mathbb{D}$, normalized by $\varphi(0) = 1$ and $\varphi'(0) > 0$, and symmetric with respect to the real axis.  For our investigation, we consider the specific choice $\varphi(z) = \left(1 + z/2 \right)^2$.\\[2mm]
In 1992, Ma and Minda~\cite{MaMinda1992} introduced a unified approach to the study of starlike and convex functions via subordination.\\[2pt]
A function $\psi$ is called a {Ma--Minda function} if it satisfies the following properties:
\begin{enumerate}
\item $\psi$ is analytic and univalent in the $\mathbb{D}$,
\item $\operatorname{Re}\, \psi(z)>0$ \mbox{for all} $z\in\mathbb{D}$,
\item $\psi(0)=1$ and $\psi'(0)>0$,
\item $\psi(\mathbb{D})$ is symmetric with respect to the real axis,
\item $\psi(\mathbb{D})$ is starlike with respect to $1$.
\end{enumerate}
Given such a $\psi$, the associated Ma--Minda subclasses are given by
\[
\mathcal{S}^*(\psi)=\left\{f\in\mathcal{S}:\ \frac{z f'(z)}{f(z)}\prec \psi(z)\right\},
\qquad
\mathcal{C}(\psi)=\left\{f\in\mathcal{S}:\ 1+\frac{z f''(z)}{f'(z)}\prec \psi(z)\right\}.
\]
By choosing specific forms of $\psi$, several well-known subclasses of $\mathcal{S}^*$ can be obtained, and many researchers have examined their geometric properties, radii of starlikeness, and coefficient bounds. In 1971, Janowski~\cite{Janowski1971} introduced the subclass
\[
\mathcal{S}^*[A,B] := \mathcal{S}^*\!\left( \frac{1 + Az}{1 + Bz} \right), \quad \text{where } -1 \leq B < A \leq 1,
\]
which provides a two-parameter family unifying many earlier examples. Building on this perspective, further notable contributions identified concrete Ma--Minda choices that yield certain well-known  subclasses: the class $\mathcal{S}_L^* := \mathcal{S}^*(\sqrt{1 + z})$ studied by Sokół and Stankiewicz~\cite{SokolStankiewicz1996}, and the cardioid-type class $\mathcal{S}_C^* := \mathcal{S}^*(1 + {4}/{3}z + {2}/{3}z^2)$ introduced by Sharma \textit{et al}.~\cite{SharmaJainRavichandran2016}. In the same vein, several other choices of $\psi$ have been investigated, the class $\mathcal{S}_e^* := \mathcal{S}^*(e^z)$ introduced by Mendiratta \textit{et al}.~\cite{MendirattaNagpalRavichandran2015}, the class  $\Delta^* := \mathcal{S}^*(z + \sqrt{1 + z^2})$ was  introduced by Raina \textit{et al}.~\cite{RainaSokol2015}, the class $\mathcal{S}_{SG}^* := \mathcal{S}^*\!\left(2/(1 + e^{-z})\right)$ considered by Goel and Kumar~\cite{GoelKumar2020}, and $\mathcal{S}_\wp^* := \mathcal{S}^*(1 + z e^z)$ investigated by Kumar and Kamaljeet~\cite{KumarKamaljeet2021}. These examples illustrate how specific analytic targets $\psi$ naturally organize families of starlike functions under subordination.\\[2mm]

\section{Certain properties of the class $\mathcal{S}^*(\varphi)$}
In this section we prove that $\varphi(z)$ satifies all the properties of Ma-Minda function.\\
It is immediate that $\varphi(0)=1$ and $\varphi'(0)>0$, and that $\varphi$ is symmetric with respect to the real axis.
To obtain the bounds, note that
\[
\left|\varphi(z)\right|=\left|1+\frac{z}{2}\right|^{2}.
\]
Since $|z|<1$, we have $1-{|z|}/{2}<\left|1+{z}/{2}\right|<1+{|z|}/{2}$, and hence
\begin{equation}\label{eq:2.1}
\frac{1}{4} \,< \left|\left(1+\frac{z}{2}\right)^{2}\right| \,< \frac{9}{4},
\qquad  z \in \mathbb{D}.
\end{equation}
Next we verify that $\varphi$ is univalent in $\mathbb{D}$. Assume that $\varphi(z_1)=\varphi(z_2)$. Then
\[
\left(1 + \frac{z_1}{2} \right)^2 = \left(1 + \frac{z_2}{2} \right)^2,
\]
so
\[
1 + \frac{z_1}{2} = \pm \left(1 + \frac{z_2}{2} \right).
\]
If $1 + {z_1}/{2} = 1 + {z_2}/{2}$, then $z_1=z_2$. If
\[
1 + \frac{z_1}{2} = -\left(1 + \frac{z_2}{2} \right),
\]
then
\[
1+\frac{z_1}{2}=-1-\frac{z_2}{2}
\]
Therefore, 
\[
z_1+z_2=-4,
\]
which is impossible for $z_1,z_2\in\mathbb{D}$ because $|z_1+z_2|\le |z_1|+|z_2|<2$.
Therefore $\varphi$ is univalent in $\mathbb{D}$.\\[2mm]
We now verify that $\varphi$ is starlike with respect to the point $\varphi(0)=1$. Since
\begin{equation}\label{eq:2.2}
\varphi'(z)=1+\frac{z}{2},
\qquad
\varphi(z)-1=\left(1+\frac{z}{2}\right)^2-1=z+\frac{z^2}{4}=z\left(1+\frac{z}{4}\right),
\end{equation}
we have
\[
\frac{z\,\varphi'(z)}{\varphi(z)-1}
=
\frac{z(1+z/2)}{z(1+z/4)}
=
\frac{1+z/2}{1+z/4}
=:M(z).
\]
A direct computation gives
\[
\frac{M(z)-1}{M(z)+1}
=
\frac{\frac{1+z/2}{1+z/4}-1}{\frac{1+z/2}{1+z/4}+1}
=
\frac{z}{8+3z}.
\]
For $z\in\mathbb{D}$, we obtain
\[
\left|\frac{z}{8+3z}\right|
\le \frac{|z|}{8-3|z|}
<\frac{1}{5}<1.
\]
Hence $\left|\frac{M(z)-1}{M(z)+1}\right|<1$, which is equivalent to $\operatorname{Re} M(z)>0$. Therefore,
\[
\operatorname{Re}\!\left(\frac{z\,\varphi'(z)}{\varphi(z)-1}\right)
=
\operatorname{Re}\!\left(\frac{1+z/2}{1+z/4}\right)>0,
\qquad z\in\mathbb{D},
\]
and consequently $\varphi$ is starlike with respect to $\varphi(0)=1$.\\[2mm]
Finally, we prove  that $\operatorname{Re}\varphi(z)>0$ for all $z\in\mathbb{D}$.
Indeed, set
\[
w=1+\frac{z}{2}=u+iv,
\qquad u=\operatorname{Re}w,\quad v=\operatorname{Im}w.
\]
Since $|z|<1$, we have
\[
|w-1|=\left|\frac{z}{2}\right|<\frac12.
\]
Consequently,
\[
u=\operatorname{Re}w=\operatorname{Re}\!\left(1+\frac{z}{2}\right)
=1+\operatorname{Re}\!\left(\frac{z}{2}\right)
\ge 1-\left|\frac{z}{2}\right|
>1-\frac12=\frac12,
\]
and also
\[
|v|=\left|\operatorname{Im}w\right|
=\left|\operatorname{Im}\!\left(\frac{z}{2}\right)\right|
\le \left|\frac{z}{2}\right|
<\frac12.
\]
Now
\[
\varphi(z)=w^{2}=(u+iv)^{2}=(u^{2}-v^{2})+i(2uv),
\]
so
\[
\operatorname{Re}\varphi(z)=u^{2}-v^{2}.
\]
Since $|v|<1/2$, we have $-v^{2}>-\left(1/2\right)^{2}$, and hence
\[
\operatorname{Re}\varphi(z)=u^{2}-v^{2}
>
u^{2}-\left(\frac12\right)^{2}.
\]
Since $u>1/2$, it follows that
\[
u^{2}-\left(\frac12\right)^{2}>0.
\]
Therefore,
\[
\operatorname{Re}\varphi(z)>0 \qquad \text{for all } z\in\mathbb{D}.
\]
Therefore, $\varphi$ satisfies the standard hypotheses of a Ma--Minda function. We denote by $\mathcal{S}^*(\varphi)$
the class of starlike functions subordinate to $\varphi$, namely
\[
\mathcal{S}^*(\varphi)
:= \left\{\, f \in \mathcal{A} : \frac{z f'(z)}{f(z)} \prec \varphi(z) \,\right\}.
\]

\section{Integral representation of the class}
Let $f \in \mathcal{S}^*(\varphi)$, then we have  ${z f'(z)}/{f(z)} \prec \varphi(z)$. Which implies that there exists a Schwarz function $w(z)$ such that $w:\mathbb{D}\to\mathbb{D}$ is analytic, $w(0)=0$, and
\[
\frac{z f'(z)}{f(z)}=\varphi\big(w(z)\big), \quad z \in \mathbb{D}.
\]
Integrating the logarithmic derivative, we obtain
\[
\log\!\left(\frac{f(z)}{z}\right)
=\int_{0}^{z}\frac{\varphi(w(\xi))-1}{\xi}\,d\xi,
\]
and therefore the integral representation
\begin{equation}\label{eq:3.1}
f(z) = z \exp\!\left( \int_{0}^{z} \frac{\varphi(w(t)) - 1}{t}\, dt \right).
\end{equation}
Assuming $w(t)=t,$ and take $\widetilde{f}=f$ for this case, we obtain
\begin{align}\label{eq:3.2}
\widetilde{f}(z)&=z \exp \int_{0}^{z} \frac{\varphi(t)-1}{t}\, dt \\
\widetilde{f}(z)&=z+ z^2+ \frac{5}{8}z^3+ \frac{7 }{24}z^4+ \cdots.
\end{align}
The next theorem gives conditions under which the M\"obius function
\[
p(z) = \frac{1 + A z}{1 + B z}
\]
is subordinate to $\varphi(z)$, when satisfying the condition in the following theorem
ensuring that $p(\mathbb{D}) \subset \varphi(\mathbb{D})$.
Geometrically, this means that the image circle of $p(z)$ lies inside the image
of $\varphi(\mathbb{D})$, which guarantees subordination.
Using the theorem in the \cite{GoelKumar2020}, we prove similar result for our class.
\begin{thm}
Let $-1 < B < A \leq 1$, and define $p(z) = \dfrac{1 + Az}{1 + Bz}$. Then \( p(z) \prec \varphi(z) \) provided that \( \dfrac{1}{4} \leq \dfrac{1 - A}{1 - B} \) and \( \dfrac{1 + A}{1 + B} \leq \dfrac{9}{4} \).
\end{thm}
\begin{pf}
Assume that $-1 < B < A \le 1$ and that
\begin{equation}\label{eq:3.3}
\frac{1}{4} \le \frac{1-A}{1-B}
\qquad\text{and}\qquad
\frac{1+A}{1+B} \le \frac{9}{4}.
\end{equation}
We prove that $p(z) \prec \varphi(z)$, where $\varphi(z) = (1+z/2)^2$.
\medskip
Since
\[
p(z) = \frac{1+Az}{1+Bz},
\]
is a Möbius transformation with real coefficients, the image of $\partial\mathbb D$
is a Euclidean circle. A standard computation shows that $p(\partial\mathbb D)$
is the circle with center
\begin{equation}\label{eq:3.4}
a = \frac{1-AB}{1-B^2}
\end{equation}
and radius
\begin{equation}\label{eq:3.5}
r = \frac{A-B}{1-B^2}.
\end{equation}
Since $p(0)=1$ lies inside this circle, the map $p$ sends the unit disk conformally onto
the disk $\mathbb B(a,r),$
\[
p(\mathbb D) = \mathbb B(a,r).
\]
Using \eqref{eq:3.4} and \eqref{eq:3.5}, we obtain
\begin{equation}\label{eq:3.6}
a-r = \frac{1-A}{1-B},
\qquad
a+r = \frac{1+A}{1+B}.
\end{equation}
Substituting \eqref{eq:3.6} into the inequalities in \eqref{eq:3.3} shows that
\eqref{eq:3.3} is equivalent to
\begin{equation}\label{eq:3.7}
\frac14 \le a-r
\qquad\text{and}\qquad
a+r \le \frac94.
\end{equation}
The pair of inequalities in \eqref{eq:3.7} is equivalent to the condition
\[
|a - 5/4| + r \le 1,
\]
hence
\begin{equation}\label{eq:3.8}
\mathbb B(a,r) \subset \mathbb B\!\left(\frac54,\,1\right).
\end{equation}
\medskip
Next we show that
\begin{equation}\label{eq:3.9}
\mathbb B\!\left(\frac54,\,1\right) \subset \varphi(\mathbb D).
\end{equation}
For $z=e^{it}$, a direct computation shows that
\[
\left|\varphi(e^{it}) - \frac54\right|^2
= \frac{9}{8} - \frac{1}{8}\cos(2t).
\]
and this quantity is at least $1$ for all $t\in\mathbb R$. Therefore the entire
curve $\varphi(\partial\mathbb D)$ lies outside the circle of radius $1$
centered at $5/4$, except at two tangency points at $t=0$ and $t=\pi$. This proves \eqref{eq:3.9}.
\medskip
Combining \eqref{eq:3.8} and \eqref{eq:3.9} gives
\[
p(\mathbb D) = \mathbb B(a,r) \subset \mathbb B\!\left(\frac54,\,1\right)
\subset \varphi(\mathbb D).
\]
Since $\varphi$ is univalent on $\mathbb D$, the map
\[
w(z) := \varphi^{-1}(p(z))
\]
is analytic on $\mathbb D$, satisfies $w(\mathbb D)\subset\mathbb D$,
and follows $w(0)=0$ because $p(0)=\varphi(0)=1$.
Thus
\[
p(z) = \varphi(w(z)) \quad \mbox{for all } z\in\mathbb D,
\]
which shows that $p \prec \varphi$.
\medskip
This completes the proof.
\end{pf}\\
We denote the class $\mathcal{C}_\gamma$ as below,
\[
\mathcal{C}_\gamma
:=
\left\{
f \in \mathcal{A} :
\operatorname{Re}\, \left(1+\frac{z f''(z)}{f'(z)}\right)>\gamma,
\quad z\in\mathbb{D}
\right\}, \quad \mbox{for}\quad
\quad 0\le \gamma<1.
\]\\
\begin{thm}
Let $f \in \mathcal{S}^*(\varphi)$. Then $f\in \mathcal{C}_{\gamma}$ in $|z|< r_{\gamma}$, where $r_\gamma$ is the least positive root of the equation $g(r)=\gamma$, and
\[
g(r)=\left(1-r-\frac{r^2}{4}\right)-\frac{r\left(1+\frac r2\right)}{\left(1-\frac r2\right)^2(1-r^2)},
\qquad \gamma\in [0,1).
\]
\end{thm}

\begin{proof}
Let $f\in\mathcal{S}^*(\varphi)$. Then
\[
\frac{z f'(z)}{f(z)} \prec \varphi(z)=\left(1+\frac z2\right)^2,
\]
so there exists an analytic Schwarz function $w:\mathbb D\to\mathbb D$ with $w(0)=0$ such that
\[
\frac{z f'(z)}{f(z)}=\varphi(w(z)) \qquad (z\in\mathbb D).
\]
Equivalently,
\[
z f'(z)=f(z)\,\varphi(w(z)).
\]
Differentiating both sides gives
\[
z f''(z)+f'(z)=f'(z)\varphi(w(z)) + f(z)\varphi'(w(z))\,w'(z),
\]
and hence
\[
1+z\frac{f''(z)}{f'(z)}
=\varphi(w(z))+\frac{f(z)}{f'(z)}\,\varphi'(w(z))\,w'(z).
\]
Taking real parts and using $\operatorname{Re}(A+B)\ge \operatorname{Re}(A)-|B|$, we get
\begin{equation}\label{eq:3.11}
\operatorname{Re}\!\left(1+z\frac{f''(z)}{f'(z)}\right)
\ge \operatorname{Re}(\varphi(w(z)))-
\left|\frac{f(z)}{f'(z)}\,\varphi'(w(z))\,w'(z)\right|.
\end{equation}
From $z f'(z)=f(z)\varphi(w(z)),$ we have
\begin{equation}\label{eq:3.12}
\frac{f(z)}{f'(z)}=\frac{z}{\varphi(w(z))}.
\end{equation}
Substituting \eqref{eq:3.12} into \eqref{eq:3.11} gives
\[
\operatorname{Re}\!\left(1+z\frac{f''(z)}{f'(z)}\right)
\ge \operatorname{Re}(\varphi(w(z)))-
\frac{|z|}{|\varphi(w(z))|}\,|\varphi'(w(z))|\,|w'(z)|.
\]
Now apply the Schwarz--Pick lemma to $w,$
\[
|w'(z)|\le \frac{1-|w(z)|^2}{1-|z|^2},
\]
therefore, we have 
\begin{equation}\label{eq:3.13}
\operatorname{Re}\!\left(1+z\frac{f''(z)}{f'(z)}\right)
\ge \operatorname{Re}(\varphi(w(z)))-
\frac{|z|}{|\varphi(w(z))|}\,|\varphi'(w(z))|\,
\frac{1-|w(z)|^2}{1-|z|^2}.
\end{equation}
Since $0\le 1-|w(z)|^2\le 1$, we further obtain
\[
\operatorname{Re}\!\left(1+z\frac{f''(z)}{f'(z)}\right)
\ge \operatorname{Re}(\varphi(w(z)))-
\frac{|z|}{|\varphi(w(z))|}\,|\varphi'(w(z))|\,
\frac{1}{1-|z|^2}.
\]
Next, for $\varphi(z)=\left(1+ z/2\right)^2$ we have $\varphi'(z)=1+ z/2$, and therefore using Schwarz lemma, we obtain
\begin{equation}\label{eq:3.14}
\left.
\begin{aligned}
	|\varphi'(w(z))|
	&=\left|1+\frac{w(z)}2\right|
	\le 1+\frac{|w(z)|}{2}
	\le 1+\frac{|z|}{2},\\[4mm]
	|\varphi(w(z))|
	&=\left|1+\frac{w(z)}2\right|^2
	\ge \left(1-\frac{|w(z)|}{2}\right)^2
	\ge \left(1-\frac{|z|}{2}\right)^2,
\end{aligned}
\right\}
\end{equation}
substituting \eqref{eq:3.14} into \eqref{eq:3.13}, and using
\[
\operatorname{Re}(\varphi(w(z)))=\operatorname{Re}\!\left(1+\frac{w(z)}2\right)^2,
\]
we obtain
\[
\operatorname{Re}\!\left(1+z\frac{f''(z)}{f'(z)}\right)
\ge
\operatorname{Re}\!\left(1+\frac{w(z)}2\right)^2
-\frac{|z|\left(1+\frac{|z|}{2}\right)}
{\left(1-\frac{|z|}{2}\right)^2(1-|z|^2)}.
\]
Now estimating the real part, 
\[
\operatorname{Re}\!\left(1+\frac{w(z)}2\right)^2
=
\operatorname{Re}\!\left(1+w(z)+\frac{w(z)^2}{4}\right)
\ge 1-|w(z)|-\frac{|w(z)|^2}{4}
\ge 1-|z|-\frac{|z|^2}{4}.
\]
Therefore, for every $z\in\mathbb D$,
\[
\operatorname{Re}\!\left(1+z\frac{f''(z)}{f'(z)}\right)
\ge
\left(1-|z|-\frac{|z|^2}{4}\right)
-\frac{|z|\left(1+\frac {|z|}2\right)}{\left(1-\frac {|z|}2\right)^2(1-|z|^2)}
=
g(|z|),
\]
where
\[
g(r)=\left(1-r-\frac{r^2}{4}\right)-\frac{r\left(1+\frac r2\right)}{\left(1-\frac r2\right)^2(1-r^2)},
\qquad 0\le r<1.
\]
It remains to prove that $g$ is strictly decreasing on $(0,1)$. Write
\[
g(r)=1-r-\frac{r^2}{4}-h(r),
\qquad
h(r):=\frac{r\left(1+\frac r2\right)}{\left(1-\frac r2\right)^2(1-r^2)}.
\]
Then
\[
g'(r)=-1-\frac r2-h'(r).
\]
A simple computation gives
\[
h'(r)=
\frac{4\bigl(-r^4-3r^3+2r^2+3r+2\bigr)}
{(2-r)^3(1-r)^2(1+r)^2}.
\]
It is easy to see that 
\[
-r^4-3r^3+2r^2+3r+2
=
(1-r^2)(r^2+3r+2)+3r^2>0
\qquad \text{for } 0<r<1,
\]
furthermore, 
\[
(2-r)^3(1-r)^2(1+r)^2>0
\qquad \text{for } 0<r<1.
\]
Thus, we have 
\[
h'(r)>0
\qquad \text{for } 0<r<1.
\]
Therefore
\[
g'(r)=-1-\frac r2-h'(r)<0
\qquad \text{for } 0<r<1,
\]
so $g$ is strictly decreasing on $(0,1)$.
Thus, for each $\gamma\in[0,1)$, the equation $g(r)=\gamma$ has a unique solution $r_\gamma\in(0,1)$.\\[2mm]
Finally, if $|z|<r_\gamma$, then since $g$ is strictly decreasing on $(0,1)$,
\[
g(|z|)>g(r_\gamma)=\gamma.
\]
Consequently,
\[
\operatorname{Re}\!\left(1+z\frac{f''(z)}{f'(z)}\right)\ge g(|z|)>\gamma.
\]
Hence
\[
\operatorname{Re}\!\left(1+z\frac{f''(z)}{f'(z)}\right)>\gamma
\qquad \text{for } |z|<r_\gamma,
\]
that is, $f\in \mathcal{C}_\gamma$ in $|z|<r_\gamma$.
This completes the proof.
\end{proof}
The following lemma will be useful in establishing the coefficient bounds for the class $\mathcal{S}^*(\varphi)$.
\begin{lem}\cite{LiberaZlotkiewicz1982}\label{lemma1}
Let $P(z)$ be a function in the Carathéodory class $\mathcal{P},$ and  $s \in \mathbb{N}$, then $|p_n -p_{n-s}\,  p_s|\leq 2,$ $n \geq s.$ for $n = 1, 2, 3, \ldots$
\end{lem}
\begin{thm}\label{Coefficients}
Let $f(z) = z + \sum_{n=2}^{\infty} a_n z^n \in \mathcal{S}^*(\varphi)$.
Then the following coefficient bounds hold
\[
|a_2| \leq 1, \quad |a_3| \leq \frac{5}{8}, \quad \text{and} \quad |a_4| \leq 0.338667.
\]
All these bounds are sharp.
\end{thm}
\begin{proof}
Let $f \in \mathcal{S}^*(\varphi)$. Then there exists a Schwarz function  $w(z)$ satisfying
\[
w(z) = \sum_{k=1}^{\infty} w_k z^k
\]
such that
\[
\frac{z f'(z)}{f(z)} = \left(1 + \frac{w(z)}{2} \right)^2, \quad z \in \mathbb{D}.
\]
Assume that
\[
w(z) = \frac{p(z) - 1}{p(z) + 1},\quad \mbox{and}  \quad p(z) = 1 + \sum_{n=1}^{\infty}p_{n}z^{n} \in \mathcal{P}.
\]
After substituting the expressions for $w(z)$, $p(z)$, and $f(z)$, and comparing coefficients, we obtain
\begin{align*}
a_2 = \frac{p_1}{2}, \qquad
a_3 = \frac{p_1^2 + 8p_2}{32}, \qquad
a_4 = \frac{32p_3 - p_1^3}{192},
\end{align*}
Using the coefficient estimate for the Carathéodory class $\mathcal{P}$, we have $|p_n| \leq 2$ for $n \geq 1$.\\[2mm]
We now compute the bounds for $a_2, a_{3}$, and $a_{4}$ respectively,
\[
|a_2| \leq 1,
\]
which is attained by the Schwarz function $w(z) = z$, giving
\[
f(z) = z \exp\left( \frac{z^2}{8} + z \right).
\]
Similarly, using Lemma~\ref{lemma1} we can easily obtain
\[
|a_3| \leq \frac{5}{8}.
\]
Which is achieved by $w(z) = z$, which also gives
\[
f(z) = z \exp\left( \frac{z^2}{8} + z \right),
\]
For the coefficient $a_{4},$ we use the relation between the coefficients of Schwarz and Carathéodory functions,
\[
w(z)=w_1 z+w_2 z^2+w_3 z^3+w_4 z^4+\cdots,\qquad w(0)=0,\qquad |w(z)|<1.
\]
Define the associated Carath\'eodory function by
\[
p(z)=\frac{1+w(z)}{1-w(z)}
=1+p_1 z+p_2 z^2+p_3 z^3+p_4 z^4+\cdots .
\]
Since
\[
\frac{1+w(z)}{1-w(z)}
=1+2\bigl(w(z)+w(z)^2+w(z)^3+\cdots\bigr),
\]
we expand termwise to obtain the coefficient relations
\[
p_1=2w_1,
\]
\[
p_2=2\bigl(w_2+w_1^2\bigr),
\]
\[
p_3=2\bigl(w_3+2w_1w_2+w_1^3\bigr),
\]
\[
p_4=2\bigl(w_4+2w_1w_3+w_2^2+3w_1^2w_2+w_1^4\bigr).
\]
which reduces $a_{4}$ to
\begin{align*}
a_{4}=\frac{1}{3}w_3+\frac{2}{3}w_1w_2+\frac{7}{24}w_1^{3}.
\end{align*}
Now, by applying Lemma~2 of Prokhorov and Szynal~\cite{ProkhorovSzynal1981}, we obtain that
\[
|a_{4}| \leq 0.338667
\]
and hence the bound for $|a_4|$ is sharp. Let
\[
w(z) = -z \cdot \frac{z+0.508001}{1+0.508001 z},
\]
giving
\[
\frac{z f'(z)}{f(z)} \,=\, \varphi(w(z)).
\]
\[
f(z)
= z
- 0.508001\,z^{2}
- 0.209676\,z^{3}
+ 0.338667\,z^{4}
+\ldots.
\]
Which proves the sharpness of the extremal bound for $a_{4}.$
\end{proof}

\section{Sharp Hankel Determinants}
In 1966, Pommerenke \cite{Pommerenke1966} introduced the concept of the Hankel determinants for the class $\mathcal{S}$, and later it was studied by others as well. The expression of the $q$th Hankel determinant for a function $f \in \mathcal{A}$, whose coefficients are given as follows:
\[
H_q(n) =
\begin{vmatrix}
a_n & a_{n+1} & \cdots & a_{n+q-1} \\
a_{n+1} & a_{n+2} & \cdots & a_{n+q} \\
\vdots & \vdots & \ddots & \vdots \\
a_{n+q-1} & a_{n+q} & \cdots & a_{n+2q-2}
\end{vmatrix}
\quad \mbox{for } q, n \in \mathbb{N}.
\]
For some special choices of $n$ and $q$,   the famous Fekete-Szegő functional, which is defined as 
\begin{equation}\label{eq:4.1}
H_2(2) = a_2 a_4 - a_3^{\,2}.
\end{equation}
is the second order Hankel determinant. For our case $a_1 := 1$, the expression of the third order Hankel determinant is given by
\begin{align*}
H_3(1) =
\begin{vmatrix}
1 & a_2 & a_3 \\
a_2 & a_3 & a_4 \\
a_3 & a_4 & a_5
\end{vmatrix}
= a_3(a_2 a_4 - a_3^2) - a_4(a_4 - a_2 a_3) + a_5(a_3 - a_2^2).
\end{align*}
The following lemma serves as a basis for establishing our main result, as it contains the well-known formulas for $p_2$, $p_3$, and $p_4$.
\begin{lem} {\label{lemma2}}\textnormal{{ \cite{LiberaZlotkiewicz1982}}}
Let $p \in \mathcal{P}$ be of the form $p(z) = 1 + \sum_{n=1}^{\infty} p_n z^n$. Then
\begin{align*}
2 p_2 &= p_1^2 + \gamma \bigl(4 - p_1^2\bigr), \\
4 p_3 &= p_1^3 + 2\bigl(4 - p_1^2\bigr)p_1 \gamma
- \bigl(4 - p_1^2\bigr)p_1 \gamma^2
+ 2\bigl(4 - p_1^2\bigr)\bigl(1 - |\gamma|^2\bigr)\eta, \\
8 p_4 &= p_1^4
+ \bigl(4 - p_1^2\bigr)\gamma\bigl(p_1^2(\gamma^2 - 3\gamma + 3) + 4\gamma\bigr)
\\
&\quad- 4\bigl(4 - p_1^2\bigr)\bigl(1 - |\gamma|^2\bigr)
\left(p_1(\gamma - 1)\eta + \overline{\gamma}\,\eta^2
- (1 - |\eta|^2)\rho\right)
\end{align*}
for some complex numbers $\gamma$, $\eta$, and $\rho$ such that $|\gamma| \leq 1$, $|\eta| \leq 1$, and $|\rho| \leq 1$.
\end{lem}

Next, we recall the following well-known result by Choi \textit{et al.} \cite{ChoiKimSugawa2007}
\begin{lem}\label{lemma3} {\cite{ChoiKimSugawa2007}}
Let $A,B,C\in\mathbb R$ and define
\begin{equation*}
Y(A,B,C):=\max_{z\in\overline{\mathbb D}}\Big(|A+Bz+Cz^{2}|+1-|z|^{2}\Big).
\end{equation*}
\begin{itemize}
\item[(i)] If $AC\ge 0$, then
\begin{align*}
	Y(A,B,C)=
	\begin{cases}
		|A|+|B|+|C|, & \text{if } |B|\ge 2\,(1-|C|),\\[6pt]
		1+|A|+\dfrac{B^{2}}{4\,(1-|C|)}, & \text{if } |B|<2\,(1-|C|).
	\end{cases}
\end{align*}
\item[(ii)] If $AC<0$, then
\begin{align*}
	Y(A,B,C)=
	\begin{cases}
		1-|A|+\dfrac{B^{2}}{4\,(1-|C|)}, & \text{if } -4AC\,(C^{2}-1)\le B^{2}\ \text{ and }\ |B|<2\,(1-|C|),\\[8pt]
		1+|A|+\dfrac{B^{2}}{4\,(1+|C|)}, & \text{if } B^{2}<\min\!\big\{\,4(1+|C|)^{2},\ -4AC\,(C^{2}-1)\big\},\\[8pt]
		R(A,B,C), & \text{otherwise,}
	\end{cases}
\end{align*}
where,
\begin{align*}
	R(A,B,C)=
	\begin{cases}
		|A|+|B|+|C|, & \text{if } |C|\big(|B|+4|A|\big)\le |AB|,\\[6pt]
		-|A|+|B|+|C|, & \text{if } |AB|\le |C|\big(|B|-4|A|\big),\\[10pt]
		\big(|A|+|C|\big)\sqrt{\,1-\dfrac{B^{2}}{4AC}\,}, & \text{otherwise.}
	\end{cases}
\end{align*}
\end{itemize}
\end{lem}
We now proceed to prove the main results of this paper.
\begin{thm}
Let $f \in \mathcal{S}^*(\varphi)$. Then
$
|H_2(2)| \leq {1}/{4}.
$
The result is sharp.
\end{thm}

\begin{proof}
Let $f \in \mathcal{S}^*(\varphi)$ be given by
\[
z\frac{f'(z)}{f(z)} \prec \varphi(z), \qquad z\in \mathbb{D}.
\]
The second Hankel determinant of order $2$ is defined as:
\[
H_2(2)=a_2a_4-a_3^2.
\]
From Theorem~\ref{Coefficients}, the coefficients of $f \in \mathcal{S}^*(\varphi)$ are related to those of $p \in \mathcal{P}$ by
\[
a_2=\frac{p_1}{2},\qquad
a_3=\frac{p_1^{2}+8p_2}{32},\qquad
a_4=\frac{32p_3-p_1^{3}}{192}.
\]
Substituting these expressions into $H_2(2)=a_2a_4-a_3^2$, we obtain
\[
H_2(2)
=\left(\frac{p_1}{2}\right)\left(\frac{32p_3-p_1^3}{192}\right)
-\left(\frac{p_1^2+8p_2}{32}\right)^2
=\frac{p_1(32p_3-p_1^3)}{384}-\frac{(p_1^2+8p_2)^2}{1024}.
\]
Hence
\begin{equation}\label{eq:4.2}
H_2(2)
=\frac{p_1p_3}{12}-\frac{p_1^4}{384}-\frac{(p_1^2+8p_2)^2}{1024}
=\frac{-11p_1^4-48p_1^2p_2+256p_1p_3-192p_2^2}{3072}.
\end{equation}
Now using Lemma~\ref{lemma2}, we substitute the parametrizations of \(p_2\) and \(p_3\) into \eqref{eq:4.2}. This gives
\begin{equation}\label{eq:4.3}
\begin{aligned}
	|H_2(2)|
	=\frac{1}{3072}\Big|\Big(&-768 \gamma^2 + 32 p_1^2 \gamma (1 + 4 \gamma)
	+ p_1^4 (-19 - 8 \gamma + 16 \gamma^2) \\
	&\qquad + 512 p_1 \eta - 128 p_1^3 \eta
	+ 128 p_1 (-4 + p_1^2)\eta |\gamma|^2 \Big)\Big|,
\end{aligned}
\end{equation}
where $ |\gamma| \le 1 $ and $ |\eta| \le 1 $.
We rewrite the expression in the form
\begin{equation}\label{eq:H2(2)}
|H_2(2)|
=
\left|
A+B\gamma+C\gamma^2+D(1-|\gamma|^2)
\right|,
\end{equation}
where
\[
A=-\frac{19p_1^4}{3072},\qquad
B=\frac{p_1^2(4-p_1^2)}{384},\qquad
C=\frac{p_1^4+8p_1^2-48}{192},\qquad
D=\frac{\eta\,p_1(4-p_1^2)}{24}.
\]
Taking moduli and using the triangle inequality in \eqref{eq:H2(2)}, we get
\[
|H_2(2)|
\le
|A+B\gamma+C\gamma^2|+|D|(1-|\gamma|^2).
\]
Since only \(|D|\) depends on \(\eta\), and the above upper bound is increasing in \(|\eta|\), we may take \(|\eta|=1\). For \(0<p_1<2\), we have
\[
|D|=\frac{p_1(4-p_1^2)}{24}>0.
\]
Thus
\[
|H_2(2)|
\le
|D|
\left(
\left|
\frac{A}{|D|}+\frac{B}{|D|}\gamma+\frac{C}{|D|}\gamma^2
\right|
+1-|\gamma|^2
\right).
\]
Define
\[
A_1=\frac{A}{|D|}
=-\frac{19p_1^3}{128(4-p_1^2)},\qquad
B_1=\frac{B}{|D|}
=\frac{p_1}{16},\qquad
C_1=\frac{C}{|D|}
=-\frac{12+p_1^2}{8p_1}.
\]
Then
\begin{align*}
|H_2(2)|
\le
|D|\left(
|A_1+B_1\gamma+C_1\gamma^2|+1-|\gamma|^2
\right).
\end{align*}
Since the class $\mathcal{S}^*(\varphi)$ is invariant under rotations, we may assume without loss of generality that \(p_1\) is real and \(p_1\in[0,2]\). We first consider the case \(0<p_1<2\). We compute
\[
A_1C_1
=
\left(-\frac{19p_1^3}{128(4-p_1^2)}\right)
\left(-\frac{12+p_1^2}{8p_1}\right)
=
\frac{19p_1^2(12+p_1^2)}{1024(4-p_1^2)}
>0.
\]
Also,
\[
|C_1|=\frac{12+p_1^2}{8p_1}\ge 1
\qquad \text{for } 0<p_1<2,
\]
and hence
\[
|B_1|\ge 2(1-|C_1|).
\]
Therefore, by lemma~\ref{lemma3},
\[
\max_{|\gamma|\le 1}
\left(
|A_1+B_1\gamma+C_1\gamma^2|+1-|\gamma|^2
\right)
=
|A_1|+|B_1|+|C_1|.
\]
Consequently,
\[
|H_2(2)|\le g_1(p_1),
\]
where
\[
g_1(p_1):=|D|\bigl(|A_1|+|B_1|+|C_1|\bigr).
\]
A direct simplification gives
\[
g_1(p_1)=\frac{768-96p_1^2-5p_1^4}{3072}.
\]
Differentiating $g_{1}(p_{1})$ with respect to $p_{1}$, we obtain
\begin{align*}
g_1'(p_1)=\frac{-192p_1-20p_1^3}{3072}\le 0
\qquad \text{for } 0<p_1<2.
\end{align*}
Hence \(g_1\) is non-increasing on \((0,2)\), and therefore
\[
\sup_{0<p_1<2} g_1(p_1)
=
\lim_{p_1\to 0^+} g_1(p_1)
=
\frac{768}{3072}
=
\frac{1}{4}.
\]
It remains to treat the boundary cases $p_1=0$ and $p_1=2$, where the normalization by $|D|$ is not valid.\\[2mm]
If $p_1=0$, then \eqref{eq:4.3} gives
\[
|H_2(2)|
=
\frac{1}{3072}|-768\gamma^2|
=
\frac{|\gamma|^2}{4}
\le \frac{1}{4}.
\]
If $p_1=2$, then \eqref{eq:4.3} reduces to
\[
|H_2(2)|
=
\frac{304}{3072}
=
\frac{19}{192}
<
\frac{1}{4}.
\]
Combining the interior and boundary cases, we conclude that
\[
|H_2(2)|\le \frac{1}{4}.
\]
Now we show that the estimate is sharp. Let the Schwarz function be \(w(z)=z^2\). Substituting this into \eqref{eq:3.1}, we obtain
\begin{align*}
f(z)
&= z\, \exp\!\left( \int_{0}^{z} \frac{\varphi(t^{2}) - 1}{t}\, dt \right) \\[2mm]
&= z\, \exp{\!\left(\frac{z^{2}}{2}+\frac{z^{4}}{16}\right)}
\in \mathcal{S}^{*}(\varphi).
\end{align*}
Expanding $f$ in the series form we obtain, 
\[
f(z)=z+\frac{1}{2}z^3+\cdots,
\]
where, 
\[
a_2=0,\qquad a_3=\frac12,\qquad a_4=0.
\]
Therefore
\[
H_2(2)=a_2a_4-a_3^2=0-\left(\frac12\right)^2=-\frac14,
\]
and hence
\begin{align*}
|H_2(2)|=\frac14.
\end{align*}
Thus the bound is sharp and is attained for $w(z)=z^2$.
\end{proof}
The following lemma which represents the coefficient relations by (see \cite{ KwonLeckoSim2018, LiberaZlotkiewicz1982})  will be useful for our next proof.
\begin{lem}\label{lem:schwarz-param}
Let
$
w(z)=\sum_{n=1}^{\infty}c_n z^n
$
be a Schwarz function, that is, $\omega$ is analytic in $\mathbb D$,
$w(0)=0$, and $|w(z)|<1$ for all $z\in\mathbb D$.
If $c_1\ge 0$, then there exist complex numbers $\gamma,\eta,\rho$
with
$
|\gamma|\le 1,\, |\eta|\le 1,\, |\rho|\le 1,
$
such that
\begin{align*}
c_2 &= (1-c_1^2)\gamma,\\[2mm]
c_3 &= (1-c_1^2)\bigl(\eta(1-|\gamma|^2)-c_1\gamma^2\bigr),\\[2mm]
c_4 &= (1-c_1^2)\Bigl(
c_1^2\gamma^3
-(1-|\gamma|^2)\bigl(2c_1\gamma\eta+\overline{\gamma}\eta^2\bigr)
+(1-|\gamma|^2)(1-|\eta|^2)\rho
\Bigr).
\end{align*}
\end{lem}
\begin{thm}\label{thm:H3-1}
Let $f\in \mathcal{S}^*(\varphi)$. Then
$
|H_3(1)|\leq {1}/{9},
$
where $H_3(1)$ denote the third Hankel determinant. Moreover, the estimate is sharp.
\end{thm}
\begin{proof}
Let $f\in \mathcal{S}^*(\varphi),$ be given by 
\begin{align*}
	\frac{z f'(z)}{f(z)}\prec \varphi(z), \quad z\in \mathbb{D}.
\end{align*}
 The third Hankel determinant is defined by:
\[
H_3(1)=
\begin{vmatrix}
	1 & a_2 & a_3 \\
	a_2 & a_3 & a_4 \\
	a_3 & a_4 & a_5
\end{vmatrix}
= a_3(a_2 a_4 - a_3^2)-a_4(a_4-a_2a_3)+a_5(a_3-a_2^2).
\]
For $f\in\mathcal{S}^*(\varphi)$ there exists a Schwarz function $w$ such that
\begin{align}\label{eq:4.4}
	\frac{z f'(z)}{f(z)}=\varphi(w(z))=\left(1+\frac{w(z)}{2}\right)^2.
\end{align}
We derive the coefficients with respect to the Schwarz coefficients.
Writing the Schwarz function in the series form such that
\[
w(z)=\sum_{n=1}^{\infty} c_n z^n,
\]
and putting the expansions of $f\in \mathcal{S}^*(\varphi)$ and $w(z)$ in \eqref{eq:4.4},
we obtain the coefficient relations
\begin{equation}
	\left.
	\begin{aligned}
		a_2 &= c_1,\\
		a_3 &= \frac{1}{8}\left(5c_1^2+4c_2\right),\\
		a_4 &= \frac{1}{24}\left(7c_1^3+16c_1c_2+8c_3\right),\\
		a_5 &= \frac{1}{384}\left(43c_1^4+184c_1^2c_2+72c_2^2+176c_1c_3+96c_4\right).
	\end{aligned}
	\right\}
\end{equation}
Substituting these values into the expansion of $H_3(1)$ yields
\begin{align}\label{eq:5.2}
	9216\, H_{3}(1)=&
	-61c_1^6 + 244c_1^4c_2 + 464c_1^3c_3 + 1088c_1c_2c_3
	- 8c_1^2(89c_2^2 + 108c_4)\notag \\
	\qquad
	&- 32(9c_2^3 + 32c_3^2 - 36c_2c_4).
\end{align}
We now use the relations between the coefficients of the Schwarz functions, which can be stated as:
\begin{align}\label{eq:5.3}
	c_2 &= (1-c_1^2)\gamma,\notag\\
	c_3 &= (1-c_1^2)\left((1-|\gamma|^2)\eta - c_1\gamma^2\right),\notag\\
	c_4 &= (1-c_1^2)\left(c_1^2\gamma^3 - (1-|\gamma|^2)\left(2c_1\gamma\eta + \overline{\gamma}\eta^2 - (1-|\eta|^2)\rho\right)\right),
\end{align}
where $|\gamma|,\,  |\eta|,\,  |\rho| \in [0, 1].$\\[2mm]
Using coefficient relation in \eqref{eq:5.3} in \eqref{eq:5.2}, we obtain
\begin{equation}\label{eq:5.4}
	9216 \, H_{3}(1)=A_{1}+B_{1} \eta +C_{1}\eta^2+D_{1}\rho
\end{equation}
where,
\begin{align*}
	A_1=&-61c_1^6-244\gamma c_1^4(-1+c_1^2)+128\gamma^4c_1^2(-1+c_1^2)^2 \\
	\quad &-8\gamma^2c_1^2\left(89-120c_1^2+31c_1^4\right)
	+32\gamma^3\left(-9-7c_1^2+14c_1^4+2c_1^6\right),\\[0.5em]
	B_1=&16\left(-1+|\gamma|^2\right)c_1(-1+c_1^2)
	\left(29c_1^2+16\gamma^2(-1+c_1^2)+\gamma(68+40c_1^2)\right),\\[0.5em]
	C_1=&-32\left(-1+|\gamma|^2\right)(-1+c_1^2)
	\Bigl(-32(-1+c_1^2)+32|\gamma|^2(-1+c_1^2) \\
	\quad &-9\left(3c_1^2+4\gamma(-1+c_1^2)\right)\overline{\gamma}\Bigr),\\[0.5em]
	D_1=&288\left(-1+|\gamma|^2\right)\left(-1+|\eta|^2\right)(-1+c_1^2)
	\left(3c_1^2+4\gamma(-1+c_1^2)\right).
\end{align*}
Since the class $\mathcal{S}^*(\varphi)$ is invariant under rotations, we may assume without loss of generality that $c_1$ is real and nonnegative. Thus
$
0\le c_1\le 1.
$
Setting
\[
c_1=p_1,\qquad x=|\gamma|,\qquad y=|\eta|,
\]
and taking modulus on both sides of \eqref{eq:5.4}, using the triangle inequality and $|\rho|\le 1$, we obtain
\[
9216\,|H_3(1)|\le H(p_1,x,y),
\]
where 
\begin{align}
	H(p_1,x,y)=&61p_1^6+244p_1^4(1-p_1^2)x
	+8p_1^2(89-120p_1^2+31p_1^4)x^2 \notag\\
	\quad
	&-32(-9-7p_1^2+14p_1^4+2p_1^6)x^3
	+128p_1^2(-1+p_1^2)^2x^4 \notag\\
	\quad
	&+16(1-x^2)p_1(1-p_1^2)\bigl(29p_1^2+16x^2(1-p_1^2)+x(68+40p_1^2)\bigr)y \notag\\
	\quad
	&+32(1-x^2)(1-p_1^2)\bigl(32(1-x^2)(1-p_1^2)+9(3p_1^2+4x(1-p_1^2))x\bigr)y^2 \notag\\
	\quad
	&+288(1-x^2)(1-y^2)(1-p_1^2)\bigl(3p_1^2+4x(1-p_1^2)\bigr).
\end{align}
Now our aim is to maximize this quantity when $\mathcal{D}:=\{(p_1, x, y)\in \mathbb{R}^3:(p_1,x,y)\in [0,1]\times[0, 1]\times[0, 1]\}.$
It is indeed easy to check that the coefficient of $y$ remains non-negative over the region $\mathcal{D}.$ Therefore we take $y=1$ for the specific term 
\[
16(1-x^2)p_1(1-p_1^2)\bigl(29p_1^2+16x^2(1-p_1^2)+x(68+40p_1^2)\bigr)y
\]
and denote this new function as $H_{1}(p_1, x, y),$ with $H(p_1, x, y)\le H_{1}(p_1, x, y).$ Since it is a quadratic polynomial in the variable $y,$ 
\[
\max_{\mathcal{D}}H_{1}(p_1, x, y)=\max_{(p_{1},x)\in [0, 1]\times[0, 1]}\left\{H_{1}(p_1, x, 0), H_{1}(p_1, x, 1)\right\}.
\]
Let us denote
\[
G_{1}(p_1, x):= H_{1}(p_1, x, 1)
\qquad\text{and}\qquad
G_{2}(p_1, x):= H_{1}(p_1, x, 0).
\]
We first deal with the function $G_{1}(p_1, x).$ In this case we prove that the function attains its maximum at the point $(0, 0),$ which is $G_{1}(0, 0)=1024$.\\[2mm]
Let
\[
F(p_1,x):=1024-G_1(p_1,x),
\qquad p:=p_1.
\]
We divide the original square \([0,1]^2\) into the four rectangles
\[
Q_1=\left[0,\frac12\right]\times\left[0,\frac12\right],\qquad
Q_2=\left[0,\frac12\right]\times\left[\frac12,1\right],
\]
\[
Q_3=\left[\frac12,1\right]\times\left[0,\frac12\right],\qquad
Q_4=\left[\frac12,1\right]\times\left[\frac12,1\right].
\]
On each rectangle \(Q=[a,b]\times[c,d]\), we set
\[
p=a+(b-a)u,\qquad x=c+(d-c)v,\qquad (u,v)\in[0,1]\times[0, 1],
\]
and write the transformed polynomial in the Bernstein basis of bidegree $(6,4)$
\[
F_Q(u,v)=\sum_{i=0}^{6}\sum_{j=0}^{4} b^{(Q)}_{ij} B_i^6(u)B_j^4(v),
\]
where,
\[
B_i^6(u)=\binom{6}{i}u^i(1-u)^{6-i},
\qquad
B_j^4(v)=\binom{4}{j}v^j(1-v)^{4-j}.
\]
Now we list every Bernstein coefficient explicitly.
For $Q_{1},$ where
\[
Q_1=\left[0,\frac12\right]\times\left[0,\frac12\right],
\]
the Bernstein coefficient matrix is
\[
\mathcal B(F,Q_1)=
\begin{pmatrix}
	0 & 0 & \frac{112}{3} & 103 & 196\\[5mm]
	0 & -\frac{34}{3} & \frac{124}{9} & \frac{415}{6} & 158\\[5mm]
	\frac{512}{15} & \frac{29}{3} & \frac{1141}{60} & \frac{3551}{60} & \frac{4123}{30}\\[5mm]
	\frac{199}{2} & \frac{1209}{20} & \frac{2453}{48} & \frac{575}{8} & \frac{5351}{40}\\[5mm]
	\frac{2834}{15} & \frac{64571}{480} & \frac{25151}{240} & \frac{16589}{160} & \frac{17441}{120}\\[5mm]
	\frac{3521}{12} & \frac{7177}{32} & \frac{50297}{288} & \frac{607}{4} & \frac{8261}{48}\\[5mm]
	\frac{25827}{64} & \frac{41391}{128} & \frac{16507}{64} & \frac{27783}{128} & \frac{13983}{64}
\end{pmatrix}.
\]
Thus, we have 
\[
\min \mathcal B(F,Q_1)=-\frac{34}{3}.
\]
For $Q_{2},$ we have
\[
Q_2=\left[0,\frac12\right]\times\left[\frac12,1\right],
\]
the Bernstein coefficient matrix is
\[
\mathcal B(F,Q_2)=
\begin{pmatrix}
	196 & 289 & \frac{1228}{3} & 556 & 736\\[5mm]
	158 & \frac{1481}{6} & \frac{3322}{9} & 528 & 736\\[5mm]
	\frac{4123}{30} & \frac{12941}{60} & \frac{19921}{60} & \frac{4937}{10} & \frac{10774}{15}\\[5mm]
	\frac{5351}{40} & \frac{7827}{40} & \frac{71689}{240} & \frac{3631}{8} & \frac{3414}{5}\\[5mm]
	\frac{17441}{120} & \frac{89761}{480} & \frac{4343}{16} & \frac{198149}{480} & \frac{38131}{60}\\[5mm]
	\frac{8261}{48} & \frac{4619}{24} & \frac{73745}{288} & \frac{12141}{32} & \frac{2353}{4}\\[5mm]
	\frac{13983}{64} & \frac{28149}{128} & \frac{16873}{64} & \frac{47025}{128} & \frac{35631}{64}
\end{pmatrix}.
\]
Thus, we have
\[
\min \mathcal B(F,Q_2)=\frac{5351}{40}.
\]
For $Q_{3}.$ where
\[
Q_3=\left[\frac12,1\right]\times\left[0,\frac12\right],
\]
the Bernstein coefficient matrix is
\[
\mathcal B(F,Q_3)=
\begin{pmatrix}
	\frac{25827}{64} & \frac{41391}{128} & \frac{16507}{64} & \frac{27783}{128} & \frac{13983}{64}\\[5mm]
	\frac{49313}{96} & \frac{27037}{64} & \frac{49133}{144} & \frac{18071}{64} & \frac{25427}{96}\\[5mm]
	\frac{151069}{240} & \frac{7963}{15} & \frac{157649}{360} & \frac{3649}{10} & \frac{26469}{80}\\[5mm]
	\frac{29659}{40} & \frac{2569}{4} & \frac{43627}{80} & \frac{9313}{20} & \frac{4209}{10}\\[5mm]
	\frac{16799}{20} & \frac{90179}{120} & \frac{59713}{90} & \frac{70729}{120} & \frac{10877}{20}\\[5mm]
	\frac{5497}{6} & \frac{10285}{12} & 798 & \frac{8975}{12} & \frac{4295}{6}\\[5mm]
	963 & 963 & 963 & 963 & 963
\end{pmatrix}.
\]
Thus, we have
\[
\min \mathcal B(F,Q_3)=\frac{27783}{128}.
\]
For $Q_{4},$ where
\[
Q_4=\left[\frac12,1\right]\times\left[\frac12,1\right],
\]
the Bernstein coefficient matrix is
\[
\mathcal B(F,Q_4)=
\begin{pmatrix}
	\frac{13983}{64} & \frac{28149}{128} & \frac{16873}{64} & \frac{47025}{128} & \frac{35631}{64}\\[5mm]
	\frac{25427}{96} & \frac{47495}{192} & \frac{2441}{9} & \frac{22743}{64} & \frac{16807}{32}\\[5mm]
	\frac{26469}{80} & \frac{11873}{40} & \frac{21727}{72} & \frac{21883}{60} & \frac{122269}{240}\\[5mm]
	\frac{4209}{10} & \frac{7523}{20} & \frac{29307}{80} & \frac{16367}{40} & \frac{21007}{40}\\[5mm]
	\frac{10877}{20} & \frac{11959}{24} & \frac{21656}{45} & \frac{20291}{40} & \frac{35521}{60}\\[5mm]
	\frac{4295}{6} & \frac{2735}{4} & \frac{2009}{3} & \frac{8191}{12} & \frac{1463}{2}\\[5mm]
	963 & 963 & 963 & 963 & 963
\end{pmatrix}.
\]
Thus, we have 
\[
\min \mathcal B(F,Q_4)=\frac{13983}{64}.
\]
Therefore, after dividing the original square into four equal rectangles, the minimum Bernstein coefficient on each one is
\[
Q_1:\ -\frac{34}{3},\qquad
Q_2:\ \frac{5351}{40},\qquad
Q_3:\ \frac{27783}{128},\qquad
Q_4:\ \frac{13983}{64}.
\]
Therefore Bernstein already proves
\[
F(p_1,x)>0
\]
on \(Q_2\), \(Q_3\), and \(Q_4\), since all Bernstein coefficients there are positive.
The only unresolved rectangle is
\[
Q_1=\left[0,\frac12\right]\times\left[0,\frac12\right],
\]
because there the minimum Bernstein coefficient is negative, precisely
$
-{34}/{3}.
$
So the next step is to subdivide only \(Q_1\) into four smaller rectangles and repeat the Bernstein calculation there.\\[2mm]
Now subdivide
\[
Q_1=\left[0,\frac12\right]\times\left[0,\frac12\right]
\]
into the four smaller rectangles
\[
Q_{11}=\left[0,\frac14\right]\times\left[0,\frac14\right],\qquad
Q_{12}=\left[0,\frac14\right]\times\left[\frac14,\frac12\right],
\]
\[
Q_{13}=\left[\frac14,\frac12\right]\times\left[0,\frac14\right],\qquad
Q_{14}=\left[\frac14,\frac12\right]\times\left[\frac14,\frac12\right].
\]
We continue to work with
\[
F(p_1,x)=1024-G_1(p_1,x),
\qquad p:=p_1.
\]
On each rectangle \(Q=[a,b]\times[c,d]\), set
\[
p=a+(b-a)u,\qquad x=c+(d-c)v,\qquad (u,v)\in[0,1]\times[0, 1],
\]
and write the transformed polynomial in the Bernstein basis of bidegree \((6,4)\):
\[
F_Q(u,v)=\sum_{i=0}^{6}\sum_{j=0}^{4} b^{(Q)}_{ij} B_i^6(u)B_j^4(v),
\]
where
\[
B_i^6(u)=\binom{6}{i}u^i(1-u)^{6-i},
\qquad
B_j^4(v)=\binom{4}{j}v^j(1-v)^{4-j}.
\]
Now we list every Bernstein coefficient explicitly.
For $Q_{11},$ where
\[
Q_{11}=\left[0,\frac14\right]\times\left[0,\frac14\right],
\]
the Bernstein coefficient matrix is\\
\[
\mathcal B(F,Q_{11})=
\begin{pmatrix}
	0 & 0 & \frac{28}{3} & \frac{215}{8} & 52\\[5mm]
	0 & -\frac{17}{6} & \frac{32}{9} & \frac{583}{32} & \frac{163}{4}\\[5mm]
	\frac{128}{15} & \frac{317}{120} & \frac{16567}{2880} & \frac{16421}{960} & \frac{2191}{60}\\[5mm]
	\frac{2019}{80} & \frac{5147}{320} & \frac{119909}{7680} & \frac{119091}{5120} & \frac{12501}{320}\\[5mm]
	\frac{2969}{60} & \frac{566747}{15360} & \frac{750769}{23040} & \frac{1111661}{30720} & \frac{92093}{1920}\\[5mm]
	\frac{30893}{384} & \frac{65853}{1024} & \frac{687673}{12288} & \frac{42439}{768} & \frac{64183}{1024}\\[5mm]
	\frac{480003}{4096} & \frac{1596537}{16384} & \frac{1392057}{16384} & \frac{1307517}{16384} & \frac{338553}{4096}
\end{pmatrix}.
\]
Thus, we have
\[
\min \mathcal B(F,Q_{11})=-\frac{17}{6}.
\]
For $Q_{12}$
\[
Q_{12}=\left[0,\frac14\right]\times\left[\frac14,\frac12\right],
\]
the Bernstein coefficient matrix is\\
\[
\mathcal B(F,Q_{12})=
\begin{pmatrix}
	52 & \frac{617}{8} & \frac{659}{6} & \frac{299}{2} & 196\\[5mm]
	\frac{163}{4} & \frac{2025}{32} & \frac{6745}{72} & \frac{3157}{24} & 177\\[5mm]
	\frac{2191}{60} & \frac{17897}{320} & \frac{240187}{2880} & \frac{18999}{160} & \frac{19483}{120}\\[5mm]
	\frac{12501}{320} & \frac{280941}{5120} & \frac{605459}{7680} & 111 & \frac{48643}{320}\\[5mm]
	\frac{92093}{1920} & \frac{367063}{6144} & \frac{183625}{2304} & \frac{332063}{3072} & \frac{55993}{384}\\[5mm]
	\frac{64183}{1024} & \frac{107671}{1536} & \frac{350787}{4096} & \frac{224847}{2048} & \frac{220651}{1536}\\[5mm]
	\frac{338553}{4096} & \frac{1400907}{16384} & \frac{1578837}{16384} & \frac{1899297}{16384} & \frac{595983}{4096}
\end{pmatrix}.
\]\\
Thus, we have
\[
\min \mathcal B(F,Q_{12})=\frac{2191}{60}.
\]
For $Q_{13},$ where
\[
Q_{13}=\left[\frac14,\frac12\right]\times\left[0,\frac14\right],
\]
the Bernstein coefficient matrix is\\
\[
\mathcal B(F,Q_{13})=
\begin{pmatrix}
	\frac{480003}{4096} & \frac{1596537}{16384} & \frac{1392057}{16384} & \frac{1307517}{16384} & \frac{338553}{4096}\\[5mm]
	\frac{945721}{6144} & \frac{1069713}{8192} & \frac{2800825}{24576} & \frac{2564503}{24576} & \frac{210187}{2048}\\[5mm]
	\frac{1005743}{5120} & \frac{10410323}{61440} & \frac{27388337}{184320} & \frac{8255597}{61440} & \frac{1964059}{15360}\\[5mm]
	\frac{624031}{2560} & \frac{2182609}{10240} & \frac{1925727}{10240} & \frac{346497}{2048} & \frac{404011}{2560}\\[5mm]
	\frac{1132141}{3840} & \frac{4004671}{15360} & \frac{10668157}{46080} & \frac{3196279}{15360} & \frac{736927}{3840}\\[5mm]
	\frac{133817}{384} & \frac{477931}{1536} & \frac{1281499}{4608} & \frac{128205}{512} & \frac{29395}{128}\\[5mm]
	\frac{25827}{64} & \frac{93045}{256} & \frac{83725}{256} & \frac{75663}{256} & \frac{17325}{64}
\end{pmatrix}.
\]\\
Thus, we have 
\[
\min \mathcal B(F,Q_{13})=\frac{1307517}{16384}.
\]
For $Q_{14},$ where
\begin{align*}
	Q_{14}=\left[\frac14,\frac12\right]\times\left[\frac14,\frac12\right],
\end{align*}\\
the Bernstein coefficient matrix is\\
\[
\mathcal B(F,Q_{14})=
\begin{pmatrix}
	\frac{338553}{4096} & \frac{1400907}{16384} & \frac{1578837}{16384} & \frac{1899297}{16384} & \frac{595983}{4096}\\[5mm]
	\frac{210187}{2048} & \frac{2479985}{24576} & \frac{877263}{8192} & \frac{999909}{8192} & \frac{905345}{6144}\\[5mm]
	\frac{1964059}{15360} & \frac{497125}{4096} & \frac{4519201}{36864} & \frac{1629815}{12288} & \frac{156895}{1024}\\[5mm]
	\frac{404011}{2560} & \frac{1499603}{10240} & \frac{1459963}{10240} & \frac{302057}{2048} & \frac{417471}{2560}\\[5mm]
	\frac{736927}{3840} & \frac{2699137}{15360} & \frac{1537061}{9216} & \frac{852653}{5120} & \frac{226571}{1280}\\[5mm]
	\frac{29395}{128} & \frac{106955}{512} & \frac{898999}{4608} & \frac{291607}{1536} & \frac{74993}{384}\\[5mm]
	\frac{17325}{64} & \frac{62937}{256} & \frac{58273}{256} & \frac{55749}{256} & \frac{13983}{64}
\end{pmatrix}.
\]\\
Thus, we have 
\[
\min \mathcal B(F,Q_{14})=\frac{338553}{4096}.
\]
Therefore, after subdividing \(Q_1\) into four smaller rectangles, the minimum Bernstein coefficient on each one is
\[
Q_{11}:\ -\frac{17}{6},\qquad
Q_{12}:\ \frac{2191}{60},\qquad
Q_{13}:\ \frac{1307517}{16384},\qquad
Q_{14}:\ \frac{338553}{4096}.
\]
So Bernstein now proves
\[
F(p_1,x)>0
\]
on \(Q_{12}\), \(Q_{13}\), and \(Q_{14}\), since all Bernstein coefficients there are positive.
The only unresolved rectangle after this second subdivision is
\[
Q_{11}=\left[0,\frac14\right]\times\left[0,\frac14\right],
\]
because there the minimum Bernstein coefficient is still negative, precisely
$
-{17}/{6}.
$
So the next step is to subdivide only \(Q_{11}\) again into four smaller rectangles and repeat the Bernstein calculation there.\\[2mm]
Now subdivide
\[
Q_{11}=\left[0,\frac14\right]\times\left[0,\frac14\right]
\]
into the four smaller rectangles
\[
Q_{111}=\left[0,\frac18\right]\times\left[0,\frac18\right],\qquad
Q_{112}=\left[0,\frac18\right]\times\left[\frac18,\frac14\right],
\]
\[
Q_{113}=\left[\frac18,\frac14\right]\times\left[0,\frac18\right],\qquad
Q_{114}=\left[\frac18,\frac14\right]\times\left[\frac18,\frac14\right].
\]
We still work with
\[
F(p,x)=1024-G_1(p,x),\qquad p:=p_1.
\]
For each rectangle \(Q=[a,b]\times[c,d]\), set
\[
p=a+(b-a)u,\qquad x=c+(d-c)v,\qquad (u,v)\in[0,1]\times[0, 1],
\]
and write the transformed polynomial in the Bernstein basis of bidegree \((6,4)\),
\[
F_Q(u,v)=\sum_{i=0}^{6}\sum_{j=0}^{4} b_{ij}^{(Q)}\,B_i^6(u)B_j^4(v),
\]
where
\[
B_i^6(u)=\binom{6}{i}u^i(1-u)^{6-i},
\qquad
B_j^4(v)=\binom{4}{j}v^j(1-v)^{4-j}.
\]
Now the Bernstein coefficient matrices are as follows.
For $Q_{111},$ where
\begin{align*}
	Q_{111}=\left[0,\frac18\right]\times\left[0,\frac18\right],
\end{align*}\\
\[
\mathcal B(F,Q_{111})=
\begin{pmatrix}
	0 & 0 & \frac{7}{3} & \frac{439}{64} & \frac{431}{32}\\[5mm]
	0 & -\frac{17}{24} & \frac{65}{72} & \frac{7225}{1536} & \frac{2713}{256}\\[5mm]
	\frac{32}{15} & \frac{661}{960} & \frac{23693}{15360} & \frac{2197}{480} & \frac{59645}{6144}\\[5mm]
	\frac{4067}{640} & \frac{21231}{5120} & \frac{1034701}{245760} & \frac{2109869}{327680} & \frac{88153}{8192}\\[5mm]
	\frac{6049}{480} & \frac{944287}{98304} & \frac{2174381}{245760} & \frac{13412569}{1310720} & \frac{26941817}{1966080}\\[5mm]
	\frac{255581}{12288} & \frac{556783}{32768} & \frac{72585779}{4718592} & \frac{16662671}{1048576} & \frac{58100225}{3145728}\\[5mm]
	\frac{8089155}{262144} & \frac{54984213}{2097152} & \frac{99545727}{4194304} & \frac{195731055}{8388608} & \frac{209583495}{8388608}
\end{pmatrix}.
\]
Thus, we have
\[
\min \mathcal B(F,Q_{111})=-\frac{17}{24}.
\]
For $Q_{112},$ where
\begin{align*}
	Q_{112}=\left[0,\frac18\right]\times\left[\frac18,\frac14\right],
\end{align*}\\
\[
\mathcal B(F,Q_{112})=
\begin{pmatrix}
	\frac{431}{32} & \frac{1285}{64} & \frac{1381}{48} & \frac{631}{16} & 52\\[5mm]
	\frac{2713}{256} & \frac{25331}{1536} & \frac{28199}{1152} & \frac{4411}{128} & \frac{371}{8}\\[5mm]
	\frac{59645}{6144} & \frac{227921}{15360} & \frac{338927}{15360} & \frac{80139}{2560} & \frac{10201}{240}\\[5mm]
	\frac{88153}{8192} & \frac{4942371}{327680} & \frac{2641727}{122880} & \frac{2452339}{81920} & \frac{103317}{2560}\\[5mm]
	\frac{26941817}{1966080} & \frac{67529561}{3932160} & \frac{7447817}{327680} & \frac{29796421}{983040} & \frac{1225589}{30720}\\[5mm]
	\frac{58100225}{3145728} & \frac{66212437}{3145728} & \frac{121259051}{4718592} & \frac{4243233}{131072} & \frac{1345615}{32768}\\[5mm]
	\frac{209583495}{8388608} & \frac{223435935}{8388608} & \frac{127250607}{4194304} & \frac{75626793}{2097152} & \frac{11484225}{262144}
\end{pmatrix}.
\]
Thus, we have
\[
\min \mathcal B(F,Q_{112})=\frac{59645}{6144}.
\]
For $Q_{113},$ where
\begin{align*}
	Q_{113}=\left[\frac18,\frac14\right]\times\left[0,\frac18\right],
\end{align*}
\[
\mathcal B(F,Q_{113})=
\begin{pmatrix}
	\frac{8089155}{262144} & \frac{54984213}{2097152} & \frac{99545727}{4194304} & \frac{195731055}{8388608} & \frac{209583495}{8388608}\\[5mm]
	\frac{16088873}{393216} & \frac{37167157}{1048576} & \frac{605568427}{18874368} & \frac{129080371}{4194304} & \frac{396349585}{12582912}\\[5mm]
	\frac{51939757}{983040} & \frac{73158895}{1572864} & \frac{1993607707}{47185920} & \frac{419448987}{10485760} & \frac{416937499}{10485760}\\[5mm]
	\frac{2180003}{32768} & \frac{77773049}{1310720} & \frac{28389403}{524288} & \frac{267038767}{5242880} & \frac{260564127}{5242880}\\[5mm]
	\frac{6708687}{81920} & \frac{145129691}{1966080} & \frac{798453347}{11796480} & \frac{499235521}{7864320} & \frac{160493139}{2621440}\\[5mm]
	\frac{2428585}{24576} & \frac{17664895}{196608} & \frac{32542313}{393216} & \frac{60985241}{786432} & \frac{58387873}{786432}\\[5mm]
	\frac{480003}{4096} & \frac{3516549}{32768} & \frac{6505143}{65536} & \frac{12193311}{131072} & \frac{11621391}{131072}
\end{pmatrix}.
\]
Thus, we have
\[
\min \mathcal B(F,Q_{113})=\frac{195731055}{8388608}.
\]
For $Q_{114},$ where
\[
Q_{114}=\left[\frac18,\frac14\right]\times\left[\frac18,\frac14\right],
\]
\[
\mathcal B(F,Q_{114})=
\begin{pmatrix}
	\frac{209583495}{8388608} & \frac{223435935}{8388608} & \frac{127250607}{4194304} & \frac{75626793}{2097152} & \frac{11484225}{262144}\\[5mm]
	\frac{396349585}{12582912} & \frac{405458057}{12582912} & \frac{660219259}{18874368} & \frac{41680929}{1048576} & \frac{6101765}{131072}\\[5mm]
	\frac{416937499}{10485760} & \frac{414426011}{10485760} & \frac{1948400923}{47185920} & \frac{354397343}{7864320} & \frac{50008423}{983040}\\[5mm]
	\frac{260564127}{5242880} & \frac{254089487}{5242880} & \frac{25799547}{524288} & \frac{68079949}{1310720} & \frac{1858157}{32768}\\[5mm]
	\frac{160493139}{2621440} & \frac{463723313}{7864320} & \frac{691916723}{11796480} & \frac{39525309}{655360} & \frac{15727231}{245760}\\[5mm]
	\frac{58387873}{786432} & \frac{18596835}{262144} & \frac{9115859}{131072} & \frac{13782067}{196608} & \frac{595285}{8192}\\[5mm]
	\frac{11621391}{131072} & \frac{11049471}{131072} & \frac{5361303}{65536} & \frac{2661729}{32768} & \frac{338553}{4096}
\end{pmatrix}.
\]
Thus, we have
\[
\min \mathcal B(F,Q_{114})=\frac{209583495}{8388608}.
\]
Therefore, after subdividing \(Q_{11}\) into four smaller rectangles, the minimum Bernstein coefficient on each one is
\[
Q_{111}:\ -\frac{17}{24},\qquad
Q_{112}:\ \frac{59645}{6144},\qquad
Q_{113}:\ \frac{195731055}{8388608},\qquad
Q_{114}:\ \frac{209583495}{8388608}.
\]
So Bernstein proves
\[
F(p_1,x)>0
\]
on \(Q_{112}\), \(Q_{113}\), and \(Q_{114}\).
The only unresolved rectangle after this third subdivision is
\[
Q_{111}=\left[0,\frac18\right]\times\left[0,\frac18\right],
\]
because there the minimum Bernstein coefficient is still negative, precisely
$
-{17}/{24}.
$
So at this stage, subdivision plus Bernstein settles every part except the tiny corner square \(Q_{111}\). That last square must be handled by a direct local estimate for \(F=1024-G_1\).\\[2mm]
On the last square
\[
Q_{111}=\left[0,\frac18\right]\times\left[0,\frac18\right],
\]
set
\[
F(p,x):=1024-G_1(p,x), \qquad p:=p_1.
\]
We prove directly that
\[
F(p,x)\ge 0 \qquad \text{for } (p,x)\in Q_{111},
\]
and in fact
\[
F(p,x)>0 \qquad \text{for } (p,x)\neq(0,0).
\]
First expand $F$ such that 
\[
\begin{aligned}
	F(p,x)
	&=-61p^6+464p^5-1024p^4-464p^3+2048p^2 \\
	&\quad+\left(244p^6+640p^5+620p^4+448p^3-864p^2-1088p\right)x \\
	&\quad+\left(-248p^6-720p^5+1856p^4+976p^3-2504p^2-256p+896\right)x^2 \\
	&\quad+\left(64p^6-640p^5-416p^4-448p^3+640p^2+1088p-288\right)x^3 \\
	&\quad+\left(-128p^6+256p^5+384p^4-512p^3-384p^2+256p+128\right)x^4.
\end{aligned}
\]
Now separate the quadratic part of $p$
\[
F(p,x)=Q(p,x)+R(p,x),
\]
where
\[
Q(p,x)=2048p^2-1088px+896x^2,
\]
and
\[
\begin{aligned}
	R(p,x)
	&=-61p^6+464p^5-1024p^4-464p^3 \\
	&\quad+\left(244p^6+640p^5+620p^4+448p^3-864p^2\right)x \\
	&\quad+\left(-248p^6-720p^5+1856p^4+976p^3-2504p^2-256p\right)x^2 \\
	&\quad+\left(64p^6-640p^5-416p^4-448p^3+640p^2+1088p-288\right)x^3 \\
	&\quad+\left(-128p^6+256p^5+384p^4-512p^3-384p^2+256p+128\right)x^4.
\end{aligned}
\]
Every monomial in \(R\) has total degree at least \(3\).
Now estimate \(Q\) from below. Since
\[
2px\le p^2+x^2,
\]
we get
\[
1088px\le 544p^2+544x^2.
\]
Therefore
\[
Q(p,x)\ge (2048-544)p^2+(896-544)x^2
=1504p^2+352x^2
\ge 352(p^2+x^2).
\]
Next estimate \(R\) on \(Q_{111}\), where
\[
0\le p\le \frac18,\qquad 0\le x\le \frac18.
\]
Write
\[
R(p,x)=\sum_{i+j\ge 3} c_{ij}p^ix^j.
\]
For each monomial \(p^ix^j\) with \(i+j\ge 3\), we have
\[
p^ix^j\le 8^{-(i+j-2)}(p^2+x^2).
\]
Indeed:
if $i \ge 2$, then
\[
p^ix^j=p^2p^{\,i-2}x^j\le 8^{-(i+j-2)}p^2\le 8^{-(i+j-2)}(p^2+x^2),
\]
if $i=1$, then $j\ge 2$, so
\[
p x^j=x^2(px^{j-2})\le 8^{-(j-1)}x^2=8^{-(i+j-2)}x^2\le 8^{-(i+j-2)}(p^2+x^2),
\]
if $i=0$, then $j\ge 3$, so
\[
x^j\le 8^{-(j-2)}x^2=8^{-(i+j-2)}x^2\le 8^{-(i+j-2)}(p^2+x^2).
\]
Hence
\[
|R(p,x)|
\le
\left(\sum_{i+j\ge 3}|c_{ij}|\,8^{-(i+j-2)}\right)(p^2+x^2).
\]
Now compute that coefficient sum explicitly from the coefficients of \(R\):
\[
\begin{aligned}
	\sum_{i+j\ge 3}|c_{ij}|\,8^{-(i+j-2)}
	&=
	\frac{61}{4096}
	+\frac{464}{512}
	+\frac{1024}{64}
	+\frac{464}{8}
	+\frac{244}{512}
	+\frac{640}{64}
	+\frac{620}{8}
	+\frac{448}{1}
	+\frac{864}{8}
	\\
	&\quad
	+\frac{248}{64}
	+\frac{720}{8}
	+\frac{1856}{1}
	+\frac{976}{1}
	+\frac{2504}{8}
	+\frac{256}{1}
	+\frac{64}{8}
	+\frac{640}{1}
	+\frac{416}{1}
	+\frac{448}{1}
	\\
	&\quad
	+\frac{640}{8}
	+\frac{1088}{1}
	+288
	+128
	+\frac{256}{8}
	+\frac{384}{1}
	+\frac{512}{1}
	+\frac{384}{8}
	+256
	+128
	\\
	&=\frac{42177473}{131072}.
\end{aligned}
\]
Therefore
\[
|R(p,x)|\le \frac{42177473}{131072}(p^2+x^2).
\]
Combining the two estimates,
\[
\begin{aligned}
	F(p,x)
	&=Q(p,x)+R(p,x)\\
	&\ge 352(p^2+x^2)-\frac{42177473}{131072}(p^2+x^2)\\
	&=\frac{3959871}{131072}(p^2+x^2).
\end{aligned}
\]
Since
\[
\frac{3959871}{131072}>0,
\]
we conclude that
\[
F(p,x)\ge 0 \qquad \text{on }Q_{111},
\]
and in fact
\[
F(p,x)>0 \qquad \text{for }(p,x)\neq(0,0).
\]
Thus
\[
1024-G_1(p_1,x)\ge 0
\qquad\text{on}\qquad
\left[0,\frac18\right]\times\left[0,\frac18\right].
\]
Together with the Bernstein positivity on all the other subrectangles, this completes the proof that
\[
1024-G_1(p_1,x)\ge 0
\qquad\text{for all }(p_1,x)\in[0,1]\times[0, 1],
\]
with equality only at
\[
(p_1,x)=(0,0).
\]
So finally,
\[
G_1(p_1,x)\le 1024
\quad\text{on}\quad [0,1]\times[0, 1],
\]
and the maximum is \(1024\), attained only at \((0,0)\).\\[2mm]
Now we deal with the second function $G_{2}(p_1, x)$.\\[2mm]
Let
\[
F(p_1,x):=G_2(p_1,x),
\qquad 0\le p_1\le 1,\quad 0\le x\le 1.
\]
For simplicity, write $p=p_{1}$. Then
\[
\begin{aligned}
	F(p,x)=\,&61 p^6 + 244 p^4 (1-p^2)x + 8p^2(89-120p^2+31p^4)x^2 \\
	&-32(-9-7p^2+14p^4+2p^6)x^3 + 128p^2(1-p^2)^2x^4 \\
	&+288(1-p^2)\bigl(3p^2+4(1-p^2)x\bigr)(1-x^2) \\
	&+16p(1-p^2)(1-x^2)\bigl(29p^2+(68+40p^2)x+16(1-p^2)x^2\bigr).
\end{aligned}
\]
After expanding and collecting like terms, we obtain
\[
\begin{aligned}
	F(p,x)=\,&128p^{6}x^{4}-64p^{6}x^{3}+248p^{6}x^{2}-244p^{6}x+61p^{6} \\
	&-256p^{5}x^{4}+640p^{5}x^{3}+720p^{5}x^{2}-640p^{5}x-464p^{5} \\
	&-256p^{4}x^{4}-1600p^{4}x^{3}-96p^{4}x^{2}+1396p^{4}x-864p^{4} \\
	&+512p^{3}x^{4}+448p^{3}x^{3}-976p^{3}x^{2}-448p^{3}x+464p^{3} \\
	&+128p^{2}x^{4}+2528p^{2}x^{3}-152p^{2}x^{2}-2304p^{2}x+864p^{2} \\
	&-256px^{4}-1088px^{3}+256px^{2}+1088px \\
	&-864x^{3}+1152x.
\end{aligned}
\]
This is a polynomial of bidegree $(6,4)$. Hence on the unit square $[0,1]\times[0, 1]$ it can be written in the Bernstein basis as
\[
F(p,x)=\sum_{i=0}^{6}\sum_{j=0}^{4} b_{ij}\,B_i^6(p)\,B_j^4(x),
\]
where
\[
B_i^6(p)=\binom{6}{i}p^i(1-p)^{6-i},
\qquad
B_j^4(x)=\binom{4}{j}x^j(1-x)^{4-j}.
\]
For the whole rectangle \([0,1]\times[0,1]\), the Bernstein coefficients are\\
\[
\bigl(b_{ij}\bigr)=
\begin{pmatrix}
	0 & 288 & 576 & 648 & 288\\[5mm]
	0 & \dfrac{1000}{3} & \dfrac{6064}{9} & 760 & 288\\[5mm]
	\dfrac{288}{5} & \dfrac{5968}{15} & \dfrac{2252}{3} & \dfrac{12772}{15} & \dfrac{5384}{15}\\[5mm]
	196 & \dfrac{2496}{5} & \dfrac{12158}{15} & 910 & \dfrac{2504}{5}\\[5mm]
	\dfrac{1904}{5} & \dfrac{3103}{5} & \dfrac{7606}{9} & \dfrac{13583}{15} & \dfrac{9284}{15}\\[5mm]
	\dfrac{1328}{3} & 607 & \dfrac{2170}{3} & \dfrac{2159}{3} & 524\\[5mm]
	61 & 61 & 61 & 61 & 61
\end{pmatrix}.
\]\\
Now we use the basic Bernstein-basis enclosure property: if
\[
F(p,x)=\sum_{i=0}^{6}\sum_{j=0}^{4} b_{ij}\,B_i^6(p)\,B_j^4(x)
\qquad \text{on }\quad  [0,1]\times[0, 1],
\]
then
\[
\min_{i,j} b_{ij}\le F(p,x)\le \max_{i,j} b_{ij}
\qquad \text{for all } (p,x)\in[0,1]\times[0, 1].
\]
From the coefficient matrix above, the largest Bernstein coefficient is
\[
\max_{0\le i\le 6,\,0\le j\le 4} b_{ij}=910.
\]
Therefore,
\[
F(p,x)\le 910
\qquad \text{for all } (p,x)\in[0,1]\times[0, 1].
\]
Equivalently,
\begin{align*}
	G_2(p_1,x)\le 910
	\qquad \text{for all } 0\le p_1\le 1,\,0\le x\le 1.
\end{align*}
Thus we obtain the desired upper bound
\[
\max_{0\le p_1\le 1,\,0\le x\le 1} G_2(p_1,x)\le 910.
\]
Combining both cases we obtain that 
\[
H_{1}(p_1, x, y)\le 1024 \quad \text{over} \quad \mathcal{D}.
\]
Therefore 
\[
|H_{3}(1)|\le \frac{1024}{9216}=\frac{1}{9}.
\]
Finally, the estimate is sharp. Taking $w(z)=z^3$ in the defining subordination,
\[
\frac{z f'(z)}{f(z)}=\left(1+\frac{z^3}{2}\right)^2,
\]
and integrating gives
\[
f(z)=z\exp\!\left(\int_0^z \frac{(1+t^3/2)^2-1}{t}\,dt\right)
=z\exp\!\left(\frac{z^3}{3}+\frac{z^6}{24}\right).
\]
For this function, $a_2=a_3=a_5=0$ and $a_4=1/3$, hence
\[
H_3(1)=-a_4^2=-\frac19,
\]
Therefore, we have $|H_3(1)|=1/9$. This completes the proof.
\end{proof}
\vspace{4mm}
\noindent\textbf{Compliance of Ethical Standards}\\[2mm]
\noindent\textbf{Conflict of Interest:}
The authors declare that there is no conflict of interest regarding the publication of this paper.\\[2mm]
\noindent\textbf{Data Availability Statement:}
Data sharing is not applicable to this article as no datasets were generated or analyzed during the current study.\\[2mm]
\noindent\textbf{Authors' Contributions:}
Both authors have made equal contributions in reading, writing, and preparing the manuscript.\\[1mm]
\noindent\textbf{Acknowledgment:}
The research of the second-named author is supported by CSIR-JRF, New Delhi, India.

\end{document}